\documentclass[a4paper,12pt,reqno]{amsart}
\usepackage{amsmath,amsthm,amsfonts,amssymb,mathtools}
\usepackage{lmodern}
\usepackage{csquotes}
\usepackage{a4wide}
\usepackage{tikz}
\usepackage{float}
\usepackage{caption}
\usepackage{indentfirst}
\usepackage{comment}
\usepackage{shortmathj}
\usepackage{subcaption}

\usepackage[style=numeric,citestyle=numeric-comp,giveninits=true,maxbibnames=9,maxcitenames=2,useprefix=true]{biblatex}
\renewbibmacro{in:}{%
  \ifentrytype{article}{%
  }{%
    \printtext{\bibstring{in}\intitlepunct}%
  }%
}
\DeclareFieldFormat[article,inbook,incollection,inproceedings,patent,thesis,unpublished]{title}{#1\isdot} % Remove the quotation marks.
\DeclareFieldFormat[article,periodical]{pages}{#1} % Remove pp.

\renewbibmacro*{volume+number+eid}{%
  \printfield{volume}%
  \setunit*{\addnbspace}% NEW (optional); there's also \addnbthinspace
  \printfield{number}%
  \setunit{\addcomma\space}%
  \printfield{eid}}
\DeclareFieldFormat[article]{number}{\mkbibparens{#1}}

\DeclareFieldFormat[article]{titlecase}{\shortifyAMSjournalname{#1}}

\usepackage[english]{babel}
\bibliography{Bib-QSD-Rumor.bib}

\usepackage{graphicx}
\usepackage[colorlinks=true, allcolors=blue]{hyperref}
\theoremstyle{plain}
\newtheorem{thm}{Theorem}[section]

\newtheorem{lem}[thm]{Lemma}

\theoremstyle{definition}
\newtheorem{defn}[thm]{Definition}

\theoremstyle{remark}
\newtheorem*{obs}{Remark}

\numberwithin{equation}{section}
\numberwithin{figure}{section}
\numberwithin{table}{section}

\newcommand{\bbZ}{\mathbb{Z}}
\newcommand{\QQ}{\boldsymbol{Q}}
\newcommand{\PP}{\boldsymbol{P}}

\newcommand{\xx}{{\overline{x}}}

\newcommand{\dado}{\ensuremath{{\,}|{\,}}}
\newcommand{\qso}{\nu^{\ast}}
\newcommand{\qsm}{\tilde{\nu}}

\newcommand{\qsir}{\hat{\nu}}
\newcommand{\tabs}{\tau_{\Delta}}
\newcommand{\tao}{\tau^{\ast}}
\newcommand{\tam}{\tilde{\tau}}

\newcommand{\tasir}{\hat{\tau}}
\newcommand{\ctm}{\tilde{S}}
\newcommand{\ctmmt}{S^*}
\newcommand{\ctmsir}{\hat{S}}
\newcommand{\mcX}{X}
\newcommand{\cam}{\gamma}
\newcommand{\ccam}{\Gamma_{(x, y)}}

\newcommand{\ctran}{\mathcal{E}}

\DeclareMathOperator{\card}{card}
\DeclareMathOperator{\DExp}{Exp}
\DeclareMathOperator{\DNormal}{Normal}

\definecolor{bordo}{rgb}{0.8,0,0.3}
\definecolor{marinho}{rgb}{0, 0, 12.5}

\begin{document}

\title[QSDs for rumor models]{On quasi-stationary distributions for stochastic rumor models}

\author[I. Ben-Ari]{Iddo Ben-Ari}
\author[E. Lebensztayn]{Elcio Lebensztayn}
\author[L. S. Santos]{Lucas Sousa Santos}

\address[I. Ben-Ari]{Department of Mathematics, University of Connecticut, Storrs, CT 06269-1009, USA.}
\email{iddo.ben-ari@uconn.edu}

\address[E. Lebensztayn, L. S. Santos]{Instituto de Matem\'atica, Estat\'istica e Computa\c{c}\~ao Cient\'ifica, Universidade Estadual de Campinas (UNICAMP), CEP 13083-859, Campinas, SP, Brasil.}
\email{lebensz@unicamp.br, l265629@dac.unicamp.br}

\thanks{This study was financed, in part, by the S\~ao Paulo Research Foundation (FAPESP), Brazil. Process Number \#2023/13453-5.
Thanks are also due to the Coordination for the Improvement of Higher Education Personnel -- CAPES, Brazil.}

\keywords{Quasi-Stationary Distributions, Stochastic rumor, Daley--Kendall, Maki--Thompson, SIR epidemic model.}
\subjclass[2020]{60F99, 60J27, 60J50.}
\date{\today}

\begin{abstract}
This paper examines the quasi-stationary behavior of stochastic rumor processes.
Using the results by \textcite{van08}, we first prove that the continuous-time Maki--Thompson model has a unique quasi-stationary distribution (QSD) given by the point mass at the state \((0, 1)\).
To obtain a non-trivial QSD, we modify the absorption set by conditioning the process on not returning to the level \(y=1\) after leaving the initial state \((N, 1)\).
For this modified model, we establish the existence and uniqueness of a non-trivial QSD that assigns positive probability to all transient states, and then derive an explicit formula for this QSD in terms of paths and transition rates.
We also discuss the ratio of expectations distribution as an alternative approach to describe the long-term behavior before absorption.
The analysis is further extended to the Daley--Kendall rumor model and the stochastic SIR epidemic model.
\end{abstract}

\maketitle

\section{Introduction}

This work presents results about the quasi-stationary behavior of stochastic rumor processes. These objects have been studied for decades in several aspects, with their foundational results published at a similar period by \textcite{darroch65,darroch67} and \textcite{daley65}, respectively.

Rumor processes were initially developed as a modification of epidemic processes, attempting to describe the spread of information through a population. The dynamics is similar to the famous SIR (Susceptible-Infected-Recovered) epidemic model, but the way an individual is removed from the infected class differs, depending on a direct interaction between individuals. This distinguishes both systems significantly, making the rumor process a distinct entity from epidemics.
Our main interest is to study the continuous-time version of the rumor model proposed by \textcite{maki73}.
We present an overview of the definitions and references about stochastic rumor systems in Section~\ref{S: Rumor}.
The interested reader is referred to \textcite[Chapter~5]{daley99} for further details.

As for quasi-stationary distributions (QSDs), the study was initiated from the necessity of understanding the long-time behavior of absorbing processes. The theory about stationary measures is not very useful for processes where a state (or a set of states) is absorbing, such as a subcritical Bienaymé--Galton--Watson branching process. Thus, finding a limit measure for the process conditioned on not being killed gives insights into a stable behavior before the dynamic finishes.
Section~\ref{S: QSD} contains a brief exposition on QSDs and some preparatory material.
A comprehensive treatment of this subject can be found in \textcite{collet13}.

Although those two subjects are not new, to our knowledge, no attention has been given to quasi-stationarity for rumor processes. We think that the first reason for this could be the fact that epidemic models are quite popular, so the efforts would naturally be put in this field. Indeed, discussions about QSDs for these models can be found, even though many are computational applications or approximations.

The second reason is possibly the fact that most results for QSDs deal with irreducible processes, as irreducibility plays an important role in guaranteeing the existence and uniqueness of the measure. Nevertheless, as we will see in Section~\ref{S: QS Rumor}, well-known results about QSDs in reducible chains allow us to show that the QSD for the Maki--Thompson rumor model has a trivial form, in the sense that all the probability mass is concentrated at a single state. 
Therefore, to proceed with the study of the existence of a non-trivial QSD for this process, we modify the absorption set.
Hence, for the modified model, we establish the existence and uniqueness of a QSD that is not concentrated at a single state and present an explicit formula for this QSD.
Also in Section~\ref{S: QS Rumor}, we study the Maki--Thompson model applying the ratio of expectations approach, another measure introduced by \textcite{darroch65,darroch67} to describe the long-term behavior of absorbing processes. 
In Section~\ref{S: SIR}, the discussion is extended to the stochastic SIR epidemic model.

\section{Rumor models}
\label{S: Rumor}

Rumor processes as treated here were first formally defined in the work of \textcite{daley65}, where the authors formulated a stochastic model for the spreading of a rumor on a finite homogeneous population. For this process, three types of mutually exclusive classes of individuals are distinguished: ignorants, spreaders, and stiflers, composing a population of size $N+1$. The dynamics occurs in the following way:
\begin{itemize}
	\item Ignorant individuals do not know the rumor and turn into spreaders whenever they receive the rumor from spreaders;
	
	\item Spreader individuals transmit the rumor until they have a frustrating experience of contacting another individual who already knows the rumor, when they turn into stiflers;
	
	\item Stifler individuals know the rumor, but do not pass it on.
\end{itemize}
As usual, we denote the number of individuals in each category at time $t$ by $X_t$, $Y_t$, and $Z_t$, respectively.
Notice that $X_t+Y_t+Z_t = N+1$ for all $t \geq 0$; therefore, it is sufficient to know two of the variables to describe the process. Suppose that the initial values are $X_0 = N$, $Y_0 = 1$, $Z_0=0$. 
To simplify the writing, the dependence on $N$ of the random variables is omitted.

Of course, there is a natural comparison between rumor processes and epidemic models, relating ignorant to susceptible individuals, spreader to infected, and stifler to removed. 
However, in epidemics, each infected individual ceases to spread the disease according to an independent recovery.
By contrast, rumor models assume that a spreader individual continues to disseminate the rumor until interacting with another informed individual.
Upon this interaction, the individual transmitting the rumor recognizes no further incentive to continue dissemination.

In the model proposed by \textcite{daley65}, which we shall abbreviate as the DK model, people interact by pairwise contacts.
The only significant interactions involve one or two spreaders; moreover, when two such individuals meet, the encounter results in both transitioning to the stifler state.
Thus, the DK model is described by a continuous-time Markov chain $\{(X_t, Y_t)\}_{t\geq0}$ with the following transitions:
\begin{equation}
	\label{F: DK Rates}
	{\allowdisplaybreaks
		\begin{array}{cc}
			\text{transition} \quad &\text{rate}\\[0.1cm]
			(-1, 1) \quad &X Y,\\[0.1cm]
			(0, -1) \quad &Y (N + 1 - X - Y),\\[0.1cm]
			(0, -2) \quad &\displaystyle\binom{Y}{2}.
	\end{array}}%
\end{equation}
In the first transition, a spreader communicates the rumor to an ignorant individual, and the latter becomes a spreader.
The other two transitions happen when a spreader meets a stifler or two spreaders meet, respectively. 
In both cases, the interaction results in the involved spreader(s) becoming stifler(s).
The mechanism is that each spreader loses its motivation to further transmit the rumor once it becomes evident that the interlocutor is already informed.

\textcite{daley65} mainly focus on the ultimate number of persons learning the rumor. Using approximations from differential equations, they find that the proportion of individuals who know that rumor is approximately $0.2032$. 
Ignoring the intervals between transitions, they study the embedded random walk and write its forward equations, but solve them only for some values of $N$.
They also discuss two natural generalizations of the rumor process.
In the first model, there is a positive probability $p$ that a spreader tells the rumor to another individual at each interaction. In addition, each spreader in a meeting with someone already informed has a probability $\alpha$ of becoming a stifler.
In the second variant of the DK model, the authors suppose that each spreader needs exactly $k \geq 1$ frustrating encounters to stop transmitting the rumor (a case known as $k$\emph{-fold stifling model}).

A major part of the literature on stochastic rumor models concerns a process formulated by \textcite{maki73} as a discrete-time simplified version of the DK model.
In the model proposed by Maki and Thompson, or MT model for brevity, the spread of the rumor occurs via directed contacts between spreaders and other people. 
Furthermore, when a spreader calls another spreader, only the one who initiated the contact turns into a stifler. 
Considering the continuous-time process, the MT model is described by the Markov chain \(\{(X_t, Y_t)\}_{t \geq 0}\)
with transition rates
\begin{equation}
	\label{F: MT Rates}
	\begin{array}{cc}
		\text{transition} \quad &\text{rate}\\[0.2cm]
		(-1, 1) \quad &X Y,\\[0.2cm]
		(0, -1) \quad &Y (N - X).
	\end{array}
\end{equation}
In the first case, an ignorant person hears the rumor from a spreader and begins to spread it.
The second transition happens when a spreader meets another spreader or a stifler. At this point, the initiating spreader loses interest in sharing the rumor and becomes a stifler.
Despite discussing the fact that rumors had been studied before and pointing out their similarities with epidemics, \textcite{maki73} do not cite the work of \textcite{daley65}, only proposing a version of the model in which when a spreader calls another spreader, both become stiflers \cite[Section~9.4.3]{maki73}.

The treatment given in \textcite{maki73} is based on calculating the conditional expectation of $X$ and $Y$ given the past state. Then, the authors use an iterative method to approximate the expected value of ignorants at the end of the process and its mean duration, obtaining the values $0.238 N$ and $1.594 N$, respectively. 
Later, by using martingale techniques, \textcite{sudbury85} established a Law of Large Numbers for the MT model, proving that, as $N \to \infty$, the ultimate proportion of individuals who remain ignorant converges in probability to 
$0.203$, hence correcting the first value.
With the aid of martingales too, \textcite{lefevre94} deduce the exact joint distribution of the final number of ignorants and the total area under the trajectory of spreaders.
For the \(k\)-fold stifling MT model, a Law of Large Numbers is derived by \textcite{duan21} and by \textcite{lebensztayn-santos25}.
Recently, for general stochastic rumor models, the asymptotic behavior of the maximum proportion of spreader individuals throughout the process is studied in \textcite{coletti25} and \textcite{lebensztayn-rodriguez25}.

Regarding the description of the stochastic fluctuations around the deterministic limit, an approximation method used in \textcite{daley65} is named by the authors as the \emph{Principle of the Diffusion of Arbitrary Constants}, a diffusion approximation for Markov population processes detailed later by \textcite{barbour72}. 
Using martingales, under the condition that the initial proportion of spreaders is not negligible, \textcite{watson88} proves a Central Limit Theorem for the terminal proportion of ignorants, when suitably rescaled, for both DK and MT models. \textcite{pittel90} generalizes this result for the DK model, allowing the case $Y_0 = 1$.
Further general limit theorems for extensions of DK and MT models are established by \textcite{lebensztayn11a,lebensztayn11b} and by \textcite{rada21}.

Due the fact that a rumor process is \emph{strictly evolutionary} (in the sense of not returning to a previously visited state) and has a finite state space, it is possible to find $p_{i,j}=P\{(X_t, Y_t)= (i,j) \dado (X_0=N, Y_0=1)\}$.
This is done by enumerating the set of paths that connect the initial state $(N,1)$ to $(i,j)$ and using probability generating functions. 
More details on this approach can be found in \textcite[Section~5.5]{daley99}. A solution for the MT model using this strategy is presented in \textcite{gani00}; however, as discussed in the paper, the explicit calculations are not feasible even for a small value of~$N$, leading again to approximations.
In \textcite{lebensztayn15}, the exact distribution of the final number of ignorants for the MT model is obtained by analyzing the embedded chain of the process.
This distribution is given in terms of the number of certain deterministic finite automata.
In addition, the author establishes a Large Deviations Principle for the corresponding proportion, with an explicit formula for the rate function.

\section{Quasi-stationary distributions}
\label{S: QSD}

Quasi-stationary distributions (QSDs) have been extensively studied since the works of \textcite{wright31} on gene frequency in finite populations and \textcite{yaglom47} on subcritical branching processes. This field has become popular due to the necessity of analyzing ``stationarity'' for absorbing processes. The term \emph{quasi-stationary distribution} was introduced by \textcite{bartlett60}, who discussed that the absorption time for a birth-and-death chain may be long enough to be relevant to consider the long-time behavior.

The key idea is that if a process is absorbed almost surely, then a stationary measure will give no information about a possible stable behavior.
Although the process will ultimately be absorbed, it may persist in a quasi-stationary state for an extended period. The quasi-stationary distribution characterizes the behavior of the process during this time.
The process is then conditioned on not being absorbed up to time $t>0$, and we study the distribution of the trajectories that survive the longest.

Let $\{\mcX_t\}_{t \geq 0}$ be a continuous-time Markov chain with countable state space $S^{\prime}$. 
We suppose that $S^{\prime}$ can be partitioned as $S^{\prime}=S\cup \Delta$, where $S$ is a transient class and $\Delta$ is an absorption set.
We define the hitting (killing) time
\begin{equation*}
	\tabs = \inf\{t>0: \mcX_t \in \Delta\}.
\end{equation*}
As usual, for any probability measure $\nu$ on $S$, we denote by $P_{\nu}$ and $E_{\nu}$ the probability and the expectation associated with the process $\mcX$ under the initial distribution $\nu$.
For any $x \in S$, we write $P_{x} = P_{\delta_x}$ and $E_{x} = E_{\delta_x}$.

\begin{defn}
	\label{D: QSD}
	Let $\nu$ be a probability measure on $S$.
	We say that $\nu$ is a \textbf{quasi-stationary distribution} (QSD) if for all $t\geq0$ and any $A \subset S$,
	\begin{equation*}
		P_{\nu}(X_t \in A \dado \tabs > t) = \nu(A).
	\end{equation*}
\end{defn}

\begin{defn}
	\label{D: QLD}
	Let $\nu$ be a probability measure on $S$.
	We say that $\nu$ is a \textbf{quasi-limiting distribution} (QLD) or Yaglom limit (after \textcite{yaglom47}) if there exists a probability measure $\mu$ on $S$ such that, for any $A \subset S$,
	\begin{equation*}
		% \label{lim_QSD}
		\lim_{t \to \infty} P_{\mu}(\mcX_t \in A \dado \tabs > t) = \nu(A).
	\end{equation*}
\end{defn}

\textcite[Proposition 1]{meleard12} establish that $\nu$ is a QSD if and only if it is a QLD. Then, QSDs are necessarily limit measures. We may have one, infinity, or no QSD associated with a process.

The first results about this subject were obtained for chains with $S$ irreducible.  \textcite{darroch65,darroch67} use Perron--Frobenius theory to prove that there exists a unique QSD for finite state processes.
The discussion on countable state spaces was introduced by \textcite{seneta66}. Very little has been published about general state space processes. The literature on QSDs is extensive, and several results were obtained for continuous-time processes.
However, as far as we know, there are no references dealing with QSDs for rumor models.
Our work is a contribution to the efforts to understand this topic; in the direction of further research, we cite \textcite{collet13}, \textcite{van13}, and \textcite{villemonais14}.

For finite state irreducible Markov chains, by the Perron--Frobenius theorem (see \textcite{serre10}), there is a unique eigenvalue with maximal real part $-\lambda$ for the rate matrix $\QQ$, associated with the eigenvectors $\nu$ and $\psi$, such that these are strictly positive, with $\nu$ assumed to be normalized. 
Writing $p_t(i,j) = P_i(\mcX_t=j)$, we recall that the matrix $\PP(t) = (p_t(i,j), i, j \in S)$ satisfies
\[ \PP(t) = e^{\QQ t} = \sum_{n\geq0}\frac{\QQ^n t^n}{n!}. \]
Hence,
\[ \nu\PP(t) = e^{-\lambda t}\nu, \, t\geq0. \]
The number $\lambda$ is called the \emph{decay parameter} of the QSD, because the killing time $\tabs$ follows an exponential distribution with parameter $\lambda$.
%Proof in collett13 (p. 19)
\begin{thm}[Exponential killing time]
	If $\nu$ is a QSD, then there exists $\lambda(\nu) \geq 0$ such that
	$$P_\nu(\tabs>t)=e^{-\lambda(\nu)t} \geq 0, \forall t\geq0$$ i.e., starting from $\nu$, $\tabs$ is exponentially distributed with parameter $\lambda(\nu)$.
\end{thm}

In fact, a necessary condition for the existence of a QSD is that the process is exponentially killed, which was first proved by \textcite{ferrari92} for birth-and-death chains. A general proof can be found in \textcite[Theorem 2.2]{collet13}. The parameter depends on $\nu$.
Now, since $P_{\nu}(\mcX_t=i)=e^{-\lambda t}\nu_i$ for any $i \in S$ and $t\geq 0$, we have that
$$\nu\PP(t) = \nu P_{\nu}(\tabs>t), \, t\geq0,$$
whence
$$P_{\nu}(\mcX_t=i \dado \tabs>t)=\nu_i, i \in S, t\geq 0,$$
which implies that $\nu$ is a QSD.

\subsection{Reducible case}\label{SS: reducible}

For reducible Markov processes, we have no guarantee that there is a unique QSD or that this probability measure exists. Most results for this scenario are presented by \textcite{van08,van09}. 
In preparation for our study of the MT model, we review some of the results of \cite{van08}.

We suppose that the set of transient states $S$ is finite, consisting of $L$ communicating classes $S_1, S_2, \dots, S_L$.
Each class $S_k$ is assigned a submatrix $\QQ_k$ of rates. First, the classes will be defined following a partial order.

\begin{defn}
	We say the class $S_i$ is \textbf{accessible} from $S_j$ if there exists a sequence of states $k_0, k_1, \dots, k_{\ell}$, such that $k_0 \in S_j$, $k_{\ell} \in S_i$, and $q(k_m, k_{m+1}) > 0$ for every $m < \ell$. 
	In this case, we write $S_i \prec S_j$.
\end{defn}

As the first hypothesis of \textcite{van08}, it is assumed that $S_i \prec S_j$ implies that $i \leq j$. 
Thus, the set of eigenvalues of the generator matrix $\QQ$ is precisely the union of the sets of eigenvalues of the individual $\QQ_k$'s. Let $-\lambda_k$ denote the unique eigenvalue of $\QQ_k$ with maximal real part. 
As a result, if we set $\lambda:= \min_{1 \leq k \leq L} \lambda_k$, then $-\lambda$ is the eigenvalue of $\QQ$ with maximal real part. Notice that we $-\lambda$ may have algebraic multiplicity greater than one, that is, it may be equal to more than one of the $\lambda_k$'s.

To guarantee the uniqueness of the QSD, we must assume that $-\lambda$ has geometric multiplicity one.
This ensures that the corresponding eigenspace has dimension one. 
Consequently, a unique QSD will be associated with the maximal eigenvalue. Now, defining $I(\lambda) := \{k:\lambda_k=\lambda\}$, we have that $\card(I(\lambda))$ is exactly the algebraic multiplicity of $-\lambda$.
Of course, if $\card(I(\lambda))=1$, then the geometric multiplicity of $-\lambda$ is also equal to one. 
For the case when $\card(I(\lambda)) > 1$, \textcite[Theorem~6]{van08} prove that a necessary and sufficient condition for $-\lambda$ to have geometric multiplicity one is the set $\{S_k, k \in I(\lambda)\}$ to be linearly ordered.
That is assumed to be true. We call $\{S_k, k \in I(\lambda)\}$ the \textbf{set of maximal classes}.

Also, defining $a(\lambda):=\min I(\lambda)$, as a result of the partial ordering imposed, the class $S_{a(\lambda)}$ is associated with the maximal eigenvalue and is the class that is accessible from all the possible other classes that also have the maximal value, that is, $S_{a(\lambda)}$ is the \textbf{last maximal class the process can reach}.

The authors show that under the hypothesis that $-\lambda$ has geometric multiplicity one, uniqueness of the QSD is guaranteed if $S_{a(\lambda)}$ is accessible from the initial distribution, that is, for any initial configuration from which it is possible to reach $S_{a(\lambda)}$ there will exist a unique QSD. In this case, $\nu$ is unique and $\nu_i>0$ if and only if the state $i$ is accessible from $S_{a(\lambda)}$.

\begin{thm}[\cite{van08}, Theorem~3]\label{thm_reduc_1}
	If $-\lambda$, the eigenvalue of $\QQ$ with maximal real part, has geometric multiplicity one, then $\mcX$ has a unique quasi-stationary distribution $\nu$ from which $S_{a(\lambda)}$ is accessible. The vector $\nu$ is the (unique, nonnegative) solution of the system $\nu \QQ = -\lambda \nu$, with $\nu$ normalized.
\end{thm}

It is crucial to limit the discussion to quasi-stationary distributions from which $S_{a(\lambda)}$ is accessible. 
Otherwise, multiple QSDs may exist, as shown in the example presented in Section~3 of \textcite{van08}.

Hence, we have to consider an initial distribution from which $S_{a(\lambda)}$ is accessible; see Figure~\ref{Fig: QSD-support} for an illustration.
Notice that, under the conditions of Theorem~\ref{thm_reduc_1}, if $S_{a(\lambda)} = S_1$ for instance, then there is a unique QSD, regardless of the initial distribution. 

The following result complements Theorem~\ref{thm_reduc_1}.

\begin{thm}[\cite{van08}, Theorem~5]\label{thm_reduc_2}
	Suppose that $-\lambda$, the eigenvalue of $\QQ$ with maximal real part, has geometric multiplicity one, and suppose that the initial distribution $\mu$ is such that $S_{a(\lambda)}$ is accessible.
	Then the following limits exist:
	\begin{align*}
		&\lim_{t \to \infty} P_{\mu}(\tabs > t + s \dado \tabs > t) = e^{-\lambda s}, \, s \geq 0,\\[0.1cm]%
		&\lim_{t \to \infty} P_{\mu}(\mcX_t = j \dado \tabs > t) = \nu_j, \, j \in S,
	\end{align*}
	where $\nu$ is the unique QSD from which $S_{a(\lambda)}$ is accessible.
\end{thm}

\begin{figure}[!htb]
	\centering
	\includegraphics[scale=0.8]{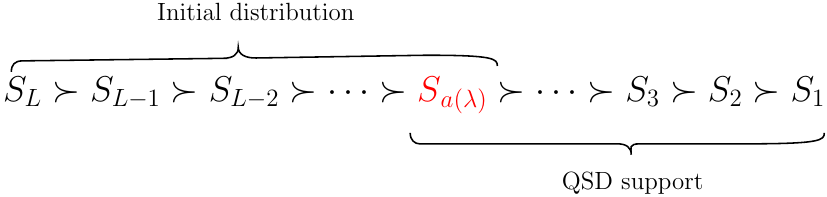}
	\caption{Example of an initial distribution and the support of the QSD according to the partial order.}
	\label{Fig: QSD-support}
\end{figure}

\subsection{Ratio of expectations approach}

The ratio of expectations (RE) distribution is another quantity that has been considered to describe a quasi-stationary behavior.
This was introduced by \textcite{ewens63} for the famous Moran genetic model, and by \textcite{darroch65,darroch67}, who deal with absorbing finite Markov chains (in discrete-time \cite{darroch65} and in continuous-time \cite{darroch67}). 
While a QSD aims to analyze the long-term behavior of an absorbing process by looking at the trajectories that survive the longest, the RE-distribution describes the dynamics before absorption in terms of the mean time spent in each state.

It may be worth reminding the reader that here $\{\mcX_t\}_{t \geq 0}$ is a continuous-time Markov chain with a countable state space $S^{\prime}=S\cup \Delta$, where $S$ is a transient class and $\Delta$ is an absorption set.

\begin{defn}
	\label{D: RE}
	Given $i, j \in S$, let
	\begin{align*}
		&T_{i}(j) = \text{time spent in $j$ for $\{\mcX_t\}_{t \geq 0}$ starting from $i$},\\[0.1cm]%
		&T_i = \sum_{j \in S} T_i(j) = \text{total time for absorption starting from $i$}.
	\end{align*}
	For a fixed initial distribution $\mu$ on $S$, we define the \textbf{ratio of expectations distribution} by
	\begin{equation}
		\label{F: RE-AIC}
		r_j(\mu) = \frac{\sum_{i \in S} \mu_i E[T_{i}(j)]}{\sum_{i \in S} \mu_i E[T_{i}]}, \, j \in S.
	\end{equation}
	For a fixed state $i \in S$, we take $\mu = \delta_i$, so that the \textbf{ratio of expectations distribution} \eqref{F: RE-AIC} reduces to
	\begin{equation}
		\label{F: RE-FIC}
		r_j = \frac{E[T_{i}(j)]}{E[T_{i}]}, \, j \in S.
	\end{equation}
\end{defn}

\textcite{darroch67} present the RE-distribution as a possible measure to describe the behavior of an absorbing process before absorption.
However, no further investigation is made, since it depends on the initial distribution. Nevertheless, the RE-distribution may provide a nice alternative for the QSD, especially when the process is absorbed too fast.

As noted in \textcite{artalejo10}, for a fixed state $i \in S$, the RE-distribution \eqref{F: RE-FIC} exists provided that $0 < E[T_i] < \infty$, whether or not $S$ is finite.
The RE-distribution assigns a positive probability to every state $j$ accessible from the initial state $i$.
This characteristic substantially distinguishes the RE-distribution from the QSD.
Furthermore, \textcite{artalejo10} consider the SIR epidemic model and compare these two approaches for understanding the behavior before absorption. 
From the results we present in Section~\ref{SS: reducible}, it turns out that all the QSD probability mass for the Markovian SIR epidemic model is concentrated at a single state; more details are given in Section \ref{S: SIR}.
On the other hand, the RE-distribution has positive values for all the states.

\textcite[Section 3.1]{meleard12} consider a continuous-time Markov chain $\{\mcX_t\}_{t \geq 0}$ with state space $\{0, 1, \dots, N\}$, $N \geq 1$, having $0$ as its unique absorbing state.
Under the hypothesis that $\mcX$ is irreducible and aperiodic before extinction, the authors argue (see Remark~3) that if the spectral gap of the generator matrix is much bigger than the decay parameter of the QSD, then the convergence to the QSD will occur in a relatively short time and, after a much longer period, the process will be absorbed.
According to \textcite{artalejo10}, a comparison between QSD and RE-distribution is justified only when the convergence to the QSD is relatively fast.
The RE-distribution method remains applicable regardless of the duration of the absorption time.
Thus, it offers a natural means to analyze the behavior of the process before absorption. 

\section{Quasi-stationary behavior of the MT model}
\label{S: QS Rumor}

In this section, we present our main results on the quasi-stationarity of the MT model.
In Section~\ref{SS: QSDs Rumor}, we apply the machinery explained in Section~\ref{SS: reducible} to prove that there is a unique QSD, which is the point mass at the state $(0, 1)$.
This probability measure is also a QLD for any initial distribution defined on the set of transient states. 
Then, striving to obtain a non-trivial QSD measure, we suitably modify the absorption set of the process. 
Thus, for the adjusted model, we prove the existence and uniqueness of a QSD that is not concentrated at a single state.
Under the adjustment, this QSD is also a QLD for the initial condition $(X_0, Y_0) = (N, 1)$.
In Section~\ref{SS: QSD Formula}, we get an explicit formula for this non-trivial QSD. 
Finally, in Section \ref{SS: RE Rumor}, we discuss the RE-distribution for the MT model.

\subsection{QSDs for the MT model}
\label{SS: QSDs Rumor}

Recall that $X_t$, $Y_t$, and $Z_t$ represent the number of people in the class of ignorant, spreader, and stifler individuals at time~$t$, respectively.
The population size is $N+1$, so that its evolution is usually described by the continuous-time Markov chain $\{(X_t, Y_t)\}_{t \geq 0}$ with transition scheme~\eqref{F: MT Rates} and initial state $(N, 1)$.
Notice that the total rate at state $(X, Y)$ equals $N Y$.
In addition, the state space, the absorbing set, and the transient set are given respectively by
\begin{align*}
	S^{\prime} &= \{ (x, y) \in \bbZ^2: 0 \leq x \leq N, 0 \leq y \leq N + 1 - x \},\\[0.1cm]%
	\Delta &= \{(x, 0): x = 0, \dots, N - 1\},\\[0.1cm]%
	S &= S^{\prime} \setminus \Delta.
\end{align*}
Also let
\[ \tau_\Delta = \inf\{t > 0: Y_t = 0\} \]
be the absorption time for the process.
Observe that the state space is finite, and the number of transitions until absorption is bounded above by $2N+1$, so absorption occurs almost surely.
Therefore, it makes sense to look for a quasi-stationary behavior. 
Figure \ref{Fig: MT-graph} shows how the process evolves, where the black points represent transient states, and the red points correspond to absorbing states. 
We note that the MT model is a \emph{strictly evolutionary process}: once it visits a transient state $(x, y) \in S$, it never comes back. Thus, using the findings from \textcite{van08}, which are presented in Section~\ref{SS: reducible}, we get our first result.

\begin{figure}[!htb]
	\centering
	\includegraphics[scale=1]{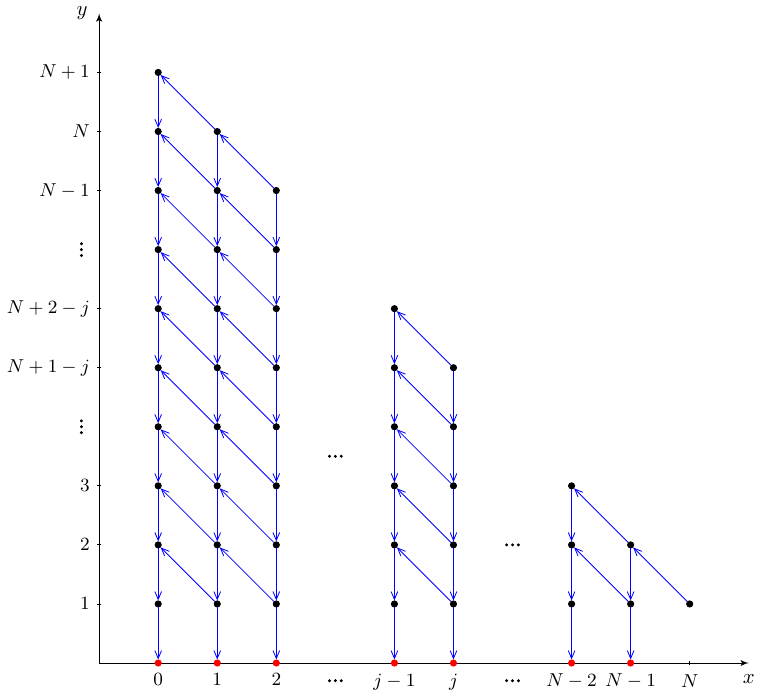}
	\caption{Possible transitions for the MT model.}
	\label{Fig: MT-graph}
\end{figure}

\begin{thm}
	\label{T: QSD-MT}
	The MT model has a unique QSD given by $\nu = \delta_{(0, 1)}$, the point mass at the state $(0, 1)$.
	Moreover, for any initial distribution $\mu$ on $S$, we have that
	\begin{equation}
		\label{F: QLD-MT}
		\lim_{t \to \infty} P_{\mu}(X_t = x, Y_t = y \dado \tau_\Delta > t) = \nu(x, y), \, (x, y) \in S.
	\end{equation}
	That is, $\nu$ is the QLD (Yaglom limit) for any initial distribution $\mu$ on $S$.
\end{thm}

\begin{proof}
	Note that, in our case, each class $S_i$ is unitary, consisting of a single state, say $i = (x, y) \in S$. 
	The total rate at which the process leaves class $S_i =\{(x, y)\}$ is $N y$, and by~\eqref{F: MT Rates} it goes from this state to another state $(x^{\prime}, y^{\prime})$ only if either $x>x^{\prime}$ and $y<y^{\prime}$, or $x=x^{\prime}$ and $y>y^{\prime}$, implying that we can organize the classes in such a way that $S_1 \prec S_2 \prec \dots \prec S_L$. 
	As the initial state is $(N, 1)$, we can take $S_L = \{(N, 1)\}$ and $S_1 = \{(0, 1)\}$, since for every state $(x, y) \in S$ there is a path to the state $(0, 1)$ if absorption has not occurred yet (see Figure \ref{Fig: MT-graph}).
	
	We can now apply Theorem~\ref{thm_reduc_1}. 
	The maximal eigenvalue for the generator matrix $\QQ$ is $-\lambda=-N$, given that the rate is minimal for states $(x, y)$ in the level $y=1$. Hence there are $N+1$ minimal classes $S_k$ such that $-\lambda_k = -N$.
	Then, defining $I(\lambda) := \{k:\lambda_k=\lambda\}$ and $a(\lambda):=\min_k I(\lambda)$, it is straightforward to see that $S_{a(\lambda)}=S_1$. 
	From Theorem \ref{thm_reduc_1}, we conclude that there exists a unique QSD $\nu$ with $\nu(x,y)>0$ if and only if $\nu(x,y)$ is accessible from the minimal class.
	However, no state is accessible from $S_{a(\lambda)}$ besides itself. 
	Consequently, $\nu = \delta_{(0, 1)}$, i.e., all the QSD probability mass is concentrated at the state $(0, 1)$.
	
	Finally, $S_{a(\lambda)}=S_1$ is accessible from any initial distribution $\mu$ on $S$, so \eqref{F: QLD-MT} readily follows from Theorem \ref{thm_reduc_2}.
\end{proof}

Intuitively, the reason why the QSD $\nu$ has a trivial form arises from the fact that the time until absorption is probably too short.
States $(x, y)$ with $y=1$ have a maximal mean sojourn time among all transient states.
Besides that, comparing states at the level $y=1$, we see that $(0, 1)$ is the ``ultimate state'', in the sense that it is accessible from all other states.
Thus, when conditioning on non-absorption, the process gets concentrated on the state associated with the minimal class, namely, $(0, 1)$.

For the SIR epidemic model, \textcite{artalejo10} prove that the QSD is also concentrated at a single state.
As the authors comment, it is inaccurate to characterize the concentration of mass at a single point as a limitation of quasi-stationarity.
In this context, the analysis of the QSD lacks relevance because the absorption time is likely not long enough.
As an alternative, \textcite{artalejo10} use the RE-distribution to describe the behavior before absorption.
We discuss this topic for the MT model in Section~\ref{SS: RE Rumor}.

From now on, we pursue our investigation of QSDs by changing the absorption time of the process.
As we show next, one way to obtain a non-trivial QSD is to modify the conditioning by treating the set of states with $y=1$ as absorbing once the process leaves the initial state $(N, 1)$ (that is, the process is conditioned on not returning to its original level).
This adjustment ensures that the process spends more time in the initial stages of the dynamics, so that the minimal class becomes $S_{a(\lambda)}=\{(N,1)\}$, and the associated QSD assigns positive probability to all transient states.

Formally, suppose that $(X_0, Y_0) = (N, 1)$, and let
\begin{equation}
	\label{F: initial_time}
	\Xi = \inf\{t>0: Y_t \neq Y_0\}
\end{equation}
be the first jump epoch of the MT model.
Then we define a new absorption time by
\begin{equation}
	\label{F: abs_time}
	\tao = \inf\{t > \Xi: Y_t = 1\},
\end{equation}
and the new set of transient states by
\begin{align*}
	\ctmmt &= S \setminus \{(x, 1): x = 0, \dots, N - 1\}\\[0.1cm]%
	&= \{(N, 1)\} \cup \{ (x, y) \in \bbZ^2: 0 \leq x \leq N, 2 \leq y \leq N + 1 - x \}.
\end{align*}
Figure~\ref{Fig: MT-Modified} illustrates the new framework: transient and absorbing states are represented by black and red dots, respectively; former absorbing states (in cyan) are now excluded.
We prove that a non-trivial QSD exists for this setting.

\begin{figure}[!htb]
	\centering
	\includegraphics[scale=1]{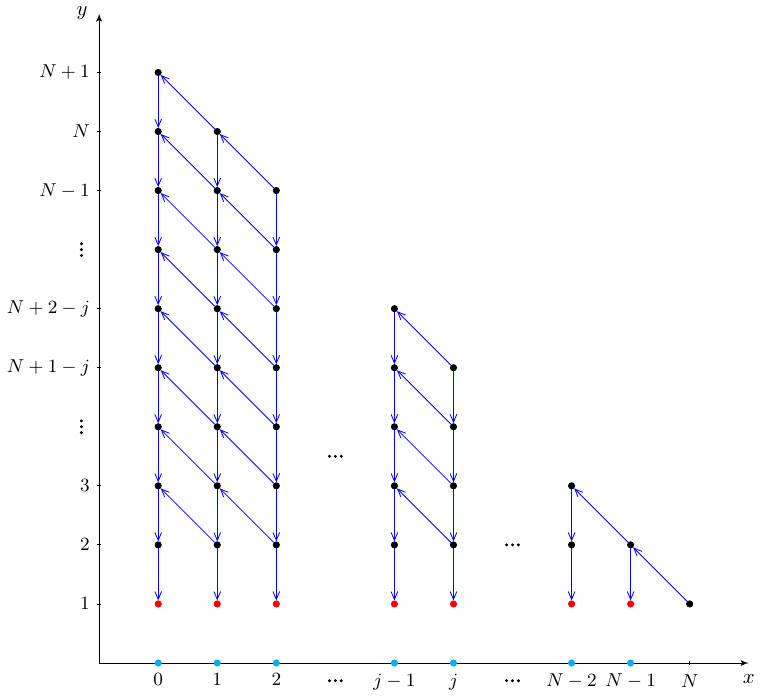}
	\caption{The modified MT model.}
	\label{Fig: MT-Modified}
\end{figure}

\begin{thm}
	\label{T: QSD-MT-2}
	Consider the hitting time defined in (\ref{F: abs_time}) for the MT model. Then there is a non-trivial measure $\qso$ such that for any $t>0$
	\begin{equation}
		\label{F: QSD-M}
		P_{\qso}(X_t=x,Y_t=y \dado \tao >t)=\qso(x,y), \, (x, y) \in \ctmmt,
	\end{equation}
	and $\qso(x,y)>0$ for every $(x, y) \in \ctmmt$. 
	This QSD satisfies
	\begin{equation}
		\label{F: QLD-M}
		\lim_{t \to \infty}P_{(N,1)}(X_t=x,Y_t=y \dado \tao >t)=\qso(x,y), \, (x, y) \in \ctmmt,
	\end{equation}
	so it is a QLD for the initial distribution $\mu=\delta_{(N,1)}$. 
	For any other initial distribution $\rho$ such that $\rho(N,1)=0$, we have that
	\begin{equation}
		\label{F: QLD-O}
		\lim_{t \to \infty}P_{\rho}(X_t=x,Y_t=y \dado \tao >t)=\delta_{(0,2)}, \, (x, y) \in \ctmmt.  
	\end{equation}
\end{thm}

\begin{proof}
	First, we note that since the process hits the level $y=1$ only one time, it follows that $\lambda=N$ and $S_{a(\lambda)}=\{(N,1)\}$. 
	From Theorem \ref{thm_reduc_1}, we conclude that there is a unique QSD $\qso$ from which $S_{a(\lambda)}$ is accessible, and $\qso(x,y)>0$ for every $(x, y) \in \ctmmt$, as all states are accessible from $\{(N, 1)\}$. This shows~\eqref{F: QSD-M}.
	
	If we take $\mu=\delta_{(N,1)}$ as the initial distribution, then $S_{a(\lambda)}$ is accessible from it, and \eqref{F: QLD-M} is an immediate consequence of Theorem~\ref{thm_reduc_2}.
	% $$\lim_{t \to \infty}P_{(N,1)}(X_t=x,Y_t=y \dado \tao >t)=\qso(x,y).$$
	Notice that, also by Theorem \ref{thm_reduc_2}, for any other initial distribution $\rho$, if $S_{a(\lambda)}$ is not accessible, then we cannot guarantee that
	$$\lim_{t \to \infty}P_{\rho}(X_t=x,Y_t=y \dado \tao >t)=\qso(x,y), \, (x, y) \in \ctmmt.$$
	If $\rho(N,1)=0$, then $\lambda=2N$, since for any $\rho \not\equiv 0$ the level $y=2$ is accessible and $S_{a(\lambda)}=\{(0, 2)\}$. 
	Applying the same argument used in the proof of Theorem \ref{T: QSD-MT}, we obtain \eqref{F: QLD-O}.
	% \[ \lim_{t \to \infty}P_{\rho}(X_t=x,Y_t=y \dado \tao >t)=\delta_{(0,2)}, \, (x, y) \in \ctmmt.\qedhere \]
\end{proof}

The main idea for Theorem~\ref{T: QSD-MT-2} is that the MT rumor model with no ignorant individuals and only one spreader can be regarded as almost extinct, because it cannot propagate further.
Note also that the process is faster when the number of spreaders is bigger. Hence, it tends to spend more time in the states \((x, y)\) where \(y=1\).
By conditioning the process on not returning to the level $y=1$, we focus the analysis on the active phase of the rumor spreading.
Within this modified framework, the initial state \((N,1)\) becomes the state with the largest mean sojourn time among all transient states.
This ensures that the resulting non-trivial QSD provides a meaningful description of the population while the rumor is effectively circulating, reflecting a balance in which the initial states have positive probability even over long periods.

\begin{obs}
	We emphasize that the choice of the absorption set is intrinsically linked to the initial configuration of the process.
	Even after modifying the absorption set, if $P_\mu((X_0, Y_0) = (N, 1)) = 0$, then the resulting QSD will be concentrated at a single state. This occurs because if the process cannot access state \((N,1)\), then it would tend to spend more time at the level \(y=2\), and then we would not force the process to spend more time in the beginning.
	This means that if the initial configuration is $(X_0, Y_0) = (x_0, y_0)$, then we will obtain a unique non-trivial QSD by conditioning on the event that the process does not return to the level $y=y_0$ after leaving it for the first time.
	This condition ensures that the process spends enough time in the beginning for the conditional limit not to concentrate in the end. 
	Formally, if $Y_0 \neq 1$ in \eqref{F: initial_time}, the absorption time to be considered in \eqref{F: abs_time} is
	\begin{equation*}
		\tao = \inf\{t > \Xi: Y_t = Y_0 \text{ and } Y_s \geq Y_0 \text{ for } 0 \leq s \leq t\}.
	\end{equation*}
	In addition, when defining the corresponding QSD, we have to condition on the event that $\infty > \tao > t$.
	We observe that $\qso$ represents the probability measure that has the closest approximation of what would be a quasi-stationary behavior for the MT model, since it removes the influence of the states with the lowest transition rates that would otherwise trigger a trivial distribution.
\end{obs}

The next step is to try to find an explicit formula for the QSD $\qso$. This issue is addressed in the following section.

\subsection{Explicit formula for the non-trivial QSD}
\label{SS: QSD Formula}

To get an explicit formula for $\qso$, we first state the following auxiliary result.

\begin{lem}\label{L: lema-sum-exp}
	
	Consider a sample $I_0, \dots, I_n$ of independent random variables such that $I_n \sim \DExp(\lambda_n)$, $n \geq 0$. Define $S_n = \sum_{i=0}^n I_i$, $n \geq 0$, and assume that $\lambda_0<\lambda_i$ for every $1 \leq i \leq n$.
	Then, for $t$ large enough, the density function of $S_n$ satisfies
	\begin{equation*}
		% \label{F: asymptotic_exp}
		f_{S_n}(t) \sim \pi_ne^{-\lambda_0 t}(1+o_n(t)),
	\end{equation*}
	where $\pi_n = \lambda_0 \prod_{j=1}^n \frac{\lambda_{j}}{\lambda_{j} - \lambda_0}, n \geq 1$ and $\pi_0 = \lambda_0$.
\end{lem}

\begin{proof}
	We proceed by induction on $n$. 
	The result is trivial for $n = 0$. 
	For $n=1$,
	\begin{align*}
		f_{S_1}(t) &= \int_0^t f_{I_0}(t-s) f_{I_{1}}(s) ds\\
		&= \int_0^t\lambda_0e^{-\lambda_0(t-s)}\lambda_1e^{-\lambda_1s}ds \\
		&= \frac{\lambda_0\lambda_1}{\lambda_1-\lambda_0} e^{-t\lambda_0}\int_0^t (\lambda_1-\lambda_0) e^{-s (\lambda_1-\lambda_0)}ds\\
		&= \pi_1 e^{-\lambda_0 t}(1-e^{-t(\lambda_1-\lambda_0)})=\pi_1 e^{-\lambda_0 t}(1+o_1(t)).
	\end{align*}
	Since $\lambda_0<\lambda_1$, the result is true for $n=1$.
	Now, assume it holds for $n$, and note that $S_{n+1}=S_n+I_{n+1}$. Then,
	\begin{align*}
		f_{S_{n+1}}(t) &= \int_0^t f_{S_{n}}(t-s) f_{I_{n+1}}(s) ds\\           &= \int_0^t\pi_ne^{-\lambda_0(t-s)}[1+o_n(t-s)]\lambda_{n+1}e^{-\lambda_{n+1}s}ds \\
		&= \pi_n\lambda_{n+1}e^{-\lambda_0 t} \int_0^t[1+o_n(t-s)]e^{-s(\lambda_{n+1}-\lambda_0)}ds.
	\end{align*}
	Now consider $M_t$ such as $t-M_t$ is small, thus
	\begin{align*}
		f_{S_{n+1}}(t) &= \pi_n \lambda_{n+1}e^{-\lambda_0t} \int_0^{t-M_t}[1+o_n(t-s)]e^{-s(\lambda_{n+1}-\lambda_0)}ds \\ 
		&\phantom{=} + \pi_n\lambda_{n+1}e^{-\lambda_0t}\int_{t-M_t}^t[1+o_n(t-s)]e^{-s(\lambda_{n+1}-\lambda_0)}ds.
	\end{align*}
	Notice that for $t>0$,
	\begin{align*}
		&\int_0^{t-M_t}[1+o_n(t-s)]e^{-s(\lambda_{n+1}-\lambda_0)}ds \sim \frac{1}{\lambda_{n+1}-\lambda_0}(1-e^{-(\lambda_{n+1}-\lambda_0)(t-M_t)})[1+o_{n+1}(M_t)],
	\end{align*}
	and
	\begin{align*}
		&\int_{t-M_t}^t[1+o_n(t-s)]e^{-s(\lambda_{n+1}-\lambda_0)}ds \sim \frac{1}{\lambda_{n+1}-\lambda_0}(e^{-(\lambda_{n+1}-\lambda_0)(t-M_t)}-e^{-(\lambda_{n+1}-\lambda_0)t})[1+o_{n+1}(M_t)].   
	\end{align*}
	Since $M_t$ is big and again $\lambda_0 < \lambda_{n+1}$, we obtain that
	\begin{align*}
		f_{S_{n+1}}(t) &\sim \frac{\pi_n \lambda_{n+1}}{\lambda_{n+1}-\lambda_0}e^{-\lambda_0t}(1-e^{-(\lambda_{n+1}-\lambda_0)t})[1+o_{n+1}(M_t)] \\
		&\sim \pi_{n+1}e^{-\lambda_0t}(1+o_{n+1}(t)).\qedhere 
	\end{align*}
\end{proof}

Before stating the main result, we use the following definitions.

\begin{defn}
	\label{D: Path}
	Let $\ctran = \{ (-1, 1), (0, -1) \}$ be the set of possible transitions of the MT model.
	For a fixed $(x,y) \in \ctmmt$, a \textbf{path} $\cam$ running from $(N, 1)$ to $(x,y)$ is a sequence of states
	\begin{equation}
		\label{F: Path}    
		v_0  = (x_0, y_0) = (N, 1), 
		v_1 = (x_1, y_1), 
		\dots, 
		v_L  = (x_L, y_L) = (x, y),
	\end{equation}
	where, for each $j \geq 0$, $v_j \in \ctmmt$ and the increment $v_{j+1} - v_j \in \ctran$.
	Let $\ccam$ denote the set of all paths from $(N, 1)$ to $(x,y)$ with all vertices in $\ctmmt$ and steps in $\ctran$.
\end{defn}

Notice that any path $\cam \in \ccam$ has $N-x$ transitions of type $(-1, 1)$, and $N-x-y+1$ transitions of type $(0, -1)$, hence
\begin{equation}
	\label{F: path_size}
	L = 2N-2x-y+1. 
\end{equation}

\begin{obs} 
	We observe that counting the number of paths in $\ccam$ is equivalent to solving the ballot problem with $a=N-x$ and $b=N+1-x-y$. See \textcite[Section III.1]{feller68} for more details. Then, for $0 \leq x \leq N$ and $2 \leq y \leq N + 1 - x$,
	\begin{equation*}
		% \label{F: N-Paths}
		|\ccam| = \dfrac{y-1}{L} \binom{L}{N-x}.
	\end{equation*}
	Figure \ref{Fig: Paths} illustrates examples of paths in $\Gamma_{(3, 5)}$ for the MT model with $N=9$.
\end{obs}

\begin{defn}
	\label{D: Rates-Path}
	Consider $\cam$ the path given by \eqref{F: Path}.
	Let $\lambda_0 = \rho_0 = N$.
	Also for each $j = 1, \dots, L-1$, we define $\lambda_j = N y_j$, and
	\begin{equation}
		\label{F: rho}
		\rho_j  = 
		\left\{
		\begin{array}{cl}
			x_j \, y_j &\text{if } v_{j+1} - v_j = (-1, 1),\\[0.2cm]
			y_j (N - x_j) &\text{if } v_{j+1} - v_j = (0, -1).
		\end{array}	\right.   
	\end{equation}
	Then, for all $j$, $\lambda_j$ is the total rate of state $v_j$, and $\rho_j$ is the infinitesimal rate associated to the transition $v_j \to v_{j+1}$.
\end{defn}

\begin{figure}[!htb]
	\centering
	\includegraphics[scale=1]{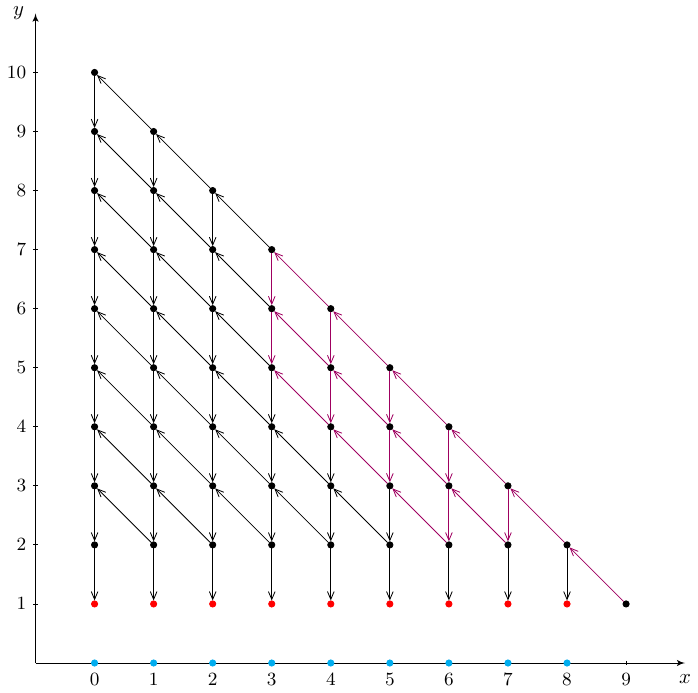}
	\caption{All paths from $(9,1)$ to $(3,5)$ for the modified MT model with $N=9$.}
	\label{Fig: Paths}
\end{figure}

The next theorem presents an explicit formula for the non-trivial QSD for the MT model.

\begin{thm}
	\label{T: QSD_MT_3}
	Consider the MT model with absorption time defined by~\eqref{F: abs_time}. 
	Then, the non-trivial QSD satisfying \eqref{F: QSD-M} and \eqref{F: QLD-M} is given by
	\begin{equation}
		\label{F: formula_QSD}
		\qso(x,y) = 
		\left\{
		\begin{array}{cl}
			C_N^{-1} &\text{if } (x, y) = (N, 1),\\[0.1cm]
			C_N^{-1}\frac{\lambda_0}{\lambda_1-\lambda0} &\text{if } (x, y) = (N-1, 2),\\[0.1cm]
			C_N^{-1} \displaystyle{\sum_{\cam \in \ccam}\frac{\lambda_0}{\lambda_L-\lambda_0} \prod_{j=1}^{L-1} \dfrac{\rho_j}{\lambda_{j} - \lambda_0}} &\text{otherwise},
		\end{array}	\right. 
	\end{equation}
	where $C_N$ is the normalizing constant defined in such a way that
	\[ \sum_{(x, y) \in \ctmmt} \qso(x,y) = 1. \]
\end{thm}

\begin{proof}
	Note that the process remains in a given state for a random exponential time before choosing a new location. Consider then $I_n \sim \DExp(\lambda_n)$, $n \geq 0$, and define a random variable $H_n$ such that for $0 \leq p_n \leq1$,
	$$P(H_n=I_n)=p_n=1-P(H_n=\infty).$$
	Now, let $I_n$ be the time the process spends at the $n$-th state and $\rho_n$ its rate as defined in \eqref{F: rho}. Surviving up to time $t>0$ means that the process reached $(x,y)$, and absorption has not occurred yet. Note that reaching $(x,y)$ means that one of the paths $\cam \in \ccam$ was chosen. Then, 
	\[ P(X_t = x, Y_t = y, \tao > t) = P\left(\{X_t = x, Y_t = y, \tao > t\} \bigcap \bigcup_{\cam \in \ccam}\{(X_t, Y_t) \text{ follows } \cam\}\right). \]
	Setting $S_L = \sum_{i=0}^L I_i$,
	{\allowdisplaybreaks
		\begin{align*}
			P(X_t = x, Y_t = y, \tao > t) &= \sum_{\cam \in \ccam}P\left(\sum_{i=0}^{L-1}H_i \leq t,I_L>t \right) \\
			&=\sum_{\cam \in \ccam}P\left(\bigcap_{i=0}^{L-1}\{H_i=I_i\},S_{L-1}\leq t, I_L > t \right)\\
			&=\sum_{\cam \in \ccam}p_0 \ p_1 \dots p_{L-1}P\left(S_{L-1}\leq t, I_L > t \right)\\
			&= \sum_{\cam \in \ccam} p_0 \ p_1 \dots p_{L-1}\int_0^tP\left(I_L > t-s \right)f_{S_{L-1}}\left(s\right) ds.
	\end{align*}}%
	Then, since $\lambda_0=N$ is the smaller rate of change, by Lemma \ref{L: lema-sum-exp}, for $t>0$ sufficiently large
	\begin{align*}
		P(X_t = x, Y_t = y, \tao > t) \sim \sum_{\cam \in \ccam} p_0 \ p_1 \dots p_{L-1}\int_0^t\pi_{L-1} e^{-\lambda_0 s}(1+o_{L-1}(s))e^{-\lambda_L(t-s)}ds
	\end{align*}
	with $\pi_{L-1} = \lambda_0 \prod_{j=1}^{L-1} \frac{\lambda_{j}}{\lambda_{j} - \lambda_0}, L \geq 2$, $\pi_0=\lambda_0$. Then,
	\begin{align*}
		P(X_t = x, Y_t = y, \tao > t) & \sim \sum_{\cam \in \ccam} \frac{p_0 \ p_1 \dots p_{L-1}}{\lambda_L}\int_0^t \pi_{L-1}e^{-\lambda_0 s}(1+o_{L-1}(s))\lambda_Le^{-\lambda_L(t-s)}ds\\&\sim  \sum_{\cam \in \ccam} \frac{p_0 \ p_1 \dots p_{L-1}}{\lambda_L}e^{-\lambda_0 t}(1+o_L(t))\lambda_0 \prod_{j=1}^{L} \frac{\lambda_{j}}{\lambda_{j} - \lambda_0} 
		\\&\sim  e^{-\lambda_0 t}(1+o_L(t))\sum_{\cam \in \ccam} \frac{\lambda_0}{\lambda_{L} - \lambda_0} \prod_{j=1}^{L-1} \frac{\rho_{j}}{\lambda_{j} - \lambda_0},
	\end{align*}
	where $\rho_j$ is defined in \eqref{F: rho}. Now, note that $$P(\tao > t) \sim e^{-\lambda_0 t}(1 + o_N(t))C_N,$$
	with $C_N$ being a constant that depends only on the population size. Therefore,
	$$\lim_{t\to\infty} P(X_t = x, Y_t = y \dado \tao > t)= C_N^{-1}\sum_{\cam \in \ccam}\frac{\lambda_0}{\lambda_{L} - \lambda_0}\prod_{j=1}^{L-1} \frac{\rho_j}{\lambda_j - \lambda_0}.$$
	By Theorem \ref{T: QSD-MT-2}, since $\mu = \delta_{(N,1)}$, \eqref{F: formula_QSD}
	is the unique non-trivial QSD for the MT model.
\end{proof}

This result provides an explicit formula for the QSD, although it involves recurrences based on the rates and probability transitions. We need to consider all the paths to state $(x,y)$ when calculating the probability, which is not an easy task for $N$ large.
The simplest probabilities are related to the case where the process never loses a spreader until attaining the state $(x, y)$ with $y = N+1-x$. 
In this setting, we can find a better expression for the probabilities, since there is only one possible path.
From \eqref{F: path_size} and \eqref{F: rho}, $L = y-1$ and $\rho_j = x_jy_j, \ j=0,\dots,L$. Then, using \eqref{F: formula_QSD}, up to the normalizing constant,
{\allowdisplaybreaks
	\begin{align*}
		\qso(x,N+1-x) &= \frac{\lambda_0}{\lambda_{L}-\lambda_0} \frac{\rho_1}{\lambda_1-\lambda_0}\frac{\rho_2}{\lambda_2-\lambda_0}\dots \frac{\rho_{L-2}}{\lambda_{L-2}-\lambda_0}\frac{\rho_{L-1}}{\lambda_{L-1}-\lambda_0}\\
		&=\frac{N}{yN-N} \frac{2(N-1)}{2N-N}\frac{3(N-2)}{3N-N}\dots \frac{(y-2)(N-y+3)}{(y-2)N-N}\frac{(y-1)(N-y+2)}{(y-1)N-N}\\
		&=\frac{1}{(y-1)} 2\frac{(N-1)}{N}\frac{3}{2}\frac{(N-2)}{N}\dots \frac{(y-2)}{(y-3)}\frac{(N-y+3)}{N}\frac{(y-1)}{(y-2)}\frac{(N-y+2)}{N}\\
		&=\frac{N-1}{N}\frac{N-2}{N}\frac{N-3}{N}\dots\frac{N-y+2}{N}.
\end{align*}}%
Therefore, since $N-y+2 = N - (N-x-1)$,
\begin{equation}\label{F: ex_qsd_1}
	\qso(x,N+1-x)= \prod_{j=1}^{N-x-1}\left(\frac{N-j}{N}\right).
\end{equation}
We can rewrite \eqref{F: ex_qsd_1} as
\begin{equation}\label{F: ex_qsd_2}
	\qso(x,N+1-x)=\frac{1}{N^{N-x-1}}\frac{(N-1)!}{x!}=\frac{1}{N^{N-x}}\frac{N!}{x!}.
\end{equation}
This gives us a clearer expression. From \eqref{F: ex_qsd_2} we can get that the probability of never losing a spreader goes to zero very fast as $N$ grows. To see this, use Stirling's formula to approximate $N!$, that is, $N! \sim \sqrt{2\pi N} (N/e)^N$. We get then
$$\qso(x,N+1-x) \sim \frac{e^{-N}}{N^{N-x}}\frac{N^N}{x!}\sqrt{2\pi N} =\sqrt{2\pi N}\frac{N^{x}e^{-N}}{x!}.$$ 
Observe that $\frac{N^{x}e^{-N}}{x!}$ is the probability mass function of a Poisson random variable with mean $N$ at the point $x$. 
For large enough $N$, this is approximately equal to the probability that a $\DNormal(N,N)$ random variable lies in the interval $[x-1/2, x+1/2]$.
Consequently,
$$\qso(x,N+1-x) \sim \sqrt{2 \pi N}\frac{1}{\sqrt{2 \pi N}}e^{-\frac{(x-N)^2}{2N}}=e^{-\frac{(x-N)^2}{2N}}.$$
Writing $\xx=x/N \in [0,1]$, we obtain for a fixed $\xx$,
\begin{equation}\label{F: ex_qsd_3}
	\qso(N\xx,N+1-N\xx) \sim e^{-N\frac{(1-\xx)^2}{2}}.  
\end{equation}
Now, it is straightforward to see that $\eqref{F: ex_qsd_3}$ goes to zero when $N$ grows. For other states, the analysis is trickier, since the transition rates change according to the paths. Figure~\ref{Fig: QSD_200} shows the graph of the QSD for the MT model with $N=200$. We can see that the graph has a peak concentrated around the state $(X=38, Y=2)$. There is also a small curve that is very close to zero, but still significant, which corresponds to the darker gray area. This suggests that there are paths that survive and paths that vanish when $N$ grows. The white area is not part of the support of the QSD.

\begin{figure}
	\begin{subfigure}{.5\textwidth}
		\centering
		\includegraphics[width=.8\linewidth]{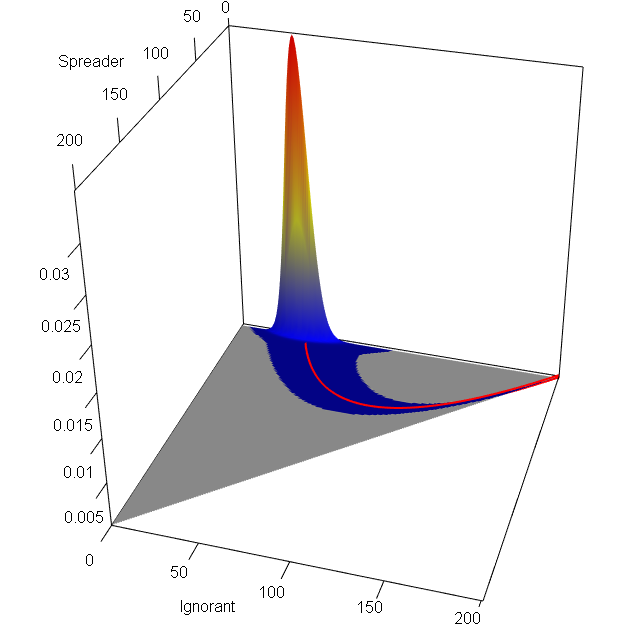}
		%\caption{1a}
	\end{subfigure}%
	\begin{subfigure}{.5\textwidth}
		\centering
		\includegraphics[width=.8\linewidth]{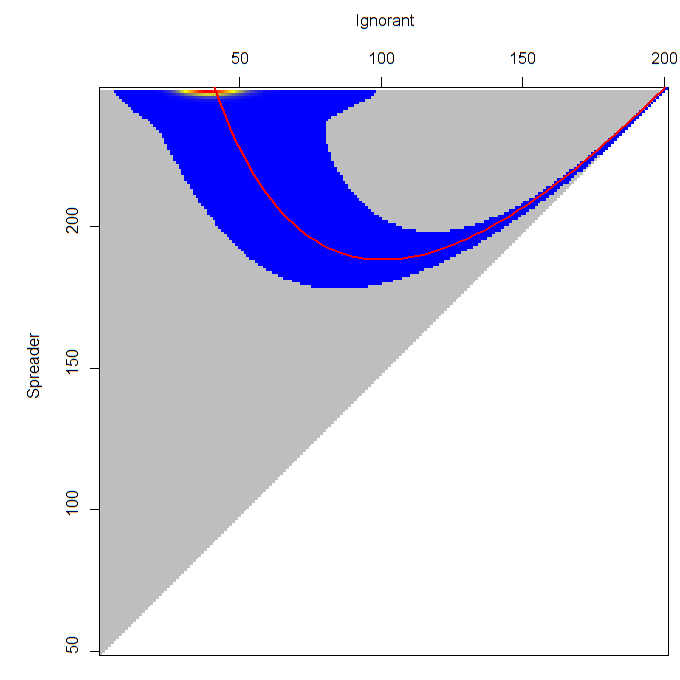}
		%\caption{1b}
	\end{subfigure}
	\caption{QSD $\qso$ for the MT model with $(X_0=200, Y_0=1)$. Values smaller than $10^{-10}$ are colored gray.}
	\label{Fig: QSD_200}
\end{figure}

The shape on the basis of the graph resembles the deterministic curve for the MT model (see \textcite{lebensztayn11a}), which is shown as the red line
on the graph. The curve is the plot of $y=f(x)=\ln(x)-2x+2, \ x \in (0,1)$, rescaled for the size of population $N+1$. The solution of $f(x)=0$ is equal to the final proportion of ignorants for the MT model, when the population size goes to infinity. 
To investigate further the behavior when $N$ goes to infinity, a better understanding of the recurrent formula would be needed. Figure \ref{Fig: Log_QSD} shows the log of the distribution for every point of the support. We can see that the curve path has a bigger probability, as compared to the rest of the support.

\begin{figure}
	\begin{subfigure}{.5\textwidth}
		\centering
		\includegraphics[width=.8\linewidth]{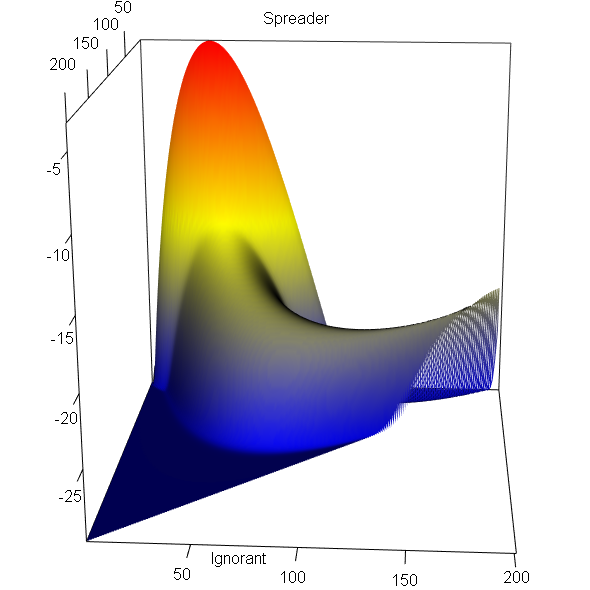}
		%\caption{1a}
	\end{subfigure}%
	\begin{subfigure}{.5\textwidth}
		\centering
		\includegraphics[width=.8\linewidth]{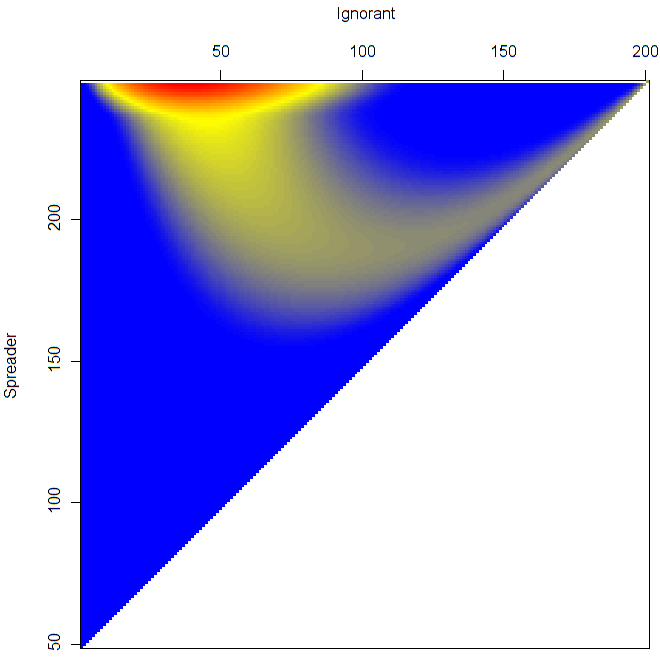}
		%\caption{1b}
	\end{subfigure}
	\caption{Log of $\qso$ for the MT model with $(X_0=200, Y_0=1)$.}
	\label{Fig: Log_QSD}
\end{figure}

\subsection{RE-distribution for the MT model}
\label{SS: RE Rumor}

In the discussion presented by \textcite{artalejo10} about the QSD and the RE-distribution for the SIR epidemic model, the authors argue that, in general, it is possible to compare these two approaches for a scenario in which the spectral gap of the generator matrix $\QQ$ is bigger than the killing rate (that is, the distance between the first and second bigger eigenvalues of $\QQ$ is bigger than $\lambda$). However, the QSD for the SIR model has a trivial form, and it is possible to check that the killing rate is greater than or equal to the spectral gap.

The line of reasoning that \textcite{artalejo10} use to prove that the SIR model has a trivial QSD is analogous to that we employ for the MT model, relying on the results from \textcite{van08}. 
Regarding the discussion about the adequate scenario to compare QSD and RE, it seems that more should be taken into consideration. 
As we show in Theorem~\ref{T: QSD-MT-2}, changing the absorption set substantially affects the existence of a non-trivial QSD.
However, the spectral gap remains the same, and in both cases it is comparable to the killing rate, as $\lambda_2-\lambda_1 = 2N-N=N=\lambda$.
For these cases, \textcite{meleard12} explain that further investigations are needed, as the set of eigenvalues and eigenfunctions has an influence.

Consider the MT model starting from the state $(X_0, Y_0) = (N, 1)$.
Then, using \eqref{F: RE-FIC} in Definition~\ref{D: RE}, the RE-distribution is given by
\begin{equation}
	\label{F: RE-MT}
	R(x,y) = \dfrac{E[T_{(x,y)}]}{E[T]},
\end{equation}
where
\begin{align*}
	&T_{(x,y)} = \text{time that the process spends in state $(x,y)$},\\[0.1cm]%
	&T = \sum_{(x,y) \in S} T_{(x,y)} = \text{total time for absorption}.
\end{align*}
We note that $R(x,y)$ given in \eqref{F: RE-MT} can be obtained recursively for a fixed $N$. 
The computations are straightforward and can be performed using mathematical software, as the process does not revisit previously visited states. 
For large $N$, we can use the asymptotic approximation for the mean time for absorption given by \textcite{svensson93}, according to which $E[T]=(2.68 \ln(N+1)+1.38)/N$.

Figure \ref{Fig: RE} shows the graph of RE-distribution for $N=200$. Naturally, this probability measure has a big peak at the beginning, since we have a small rate of change and high probabilities. Another peak can be observed near absorption, indicating paths with high probability. Again, we have a curved path that has more probability mass.
This indicates that the process spends more time in the beginning, then it accelerates until it gets slower again near absorption.
We emphasize that QSDs and the RE-distribution are distinct measures designed to describe the long-term behavior of absorbing processes. 
Their comparison must take into consideration the proper scenarios, as previously pointed out.

\begin{figure}
	\begin{subfigure}{.5\textwidth}
		\centering
		\includegraphics[width=.8\linewidth]{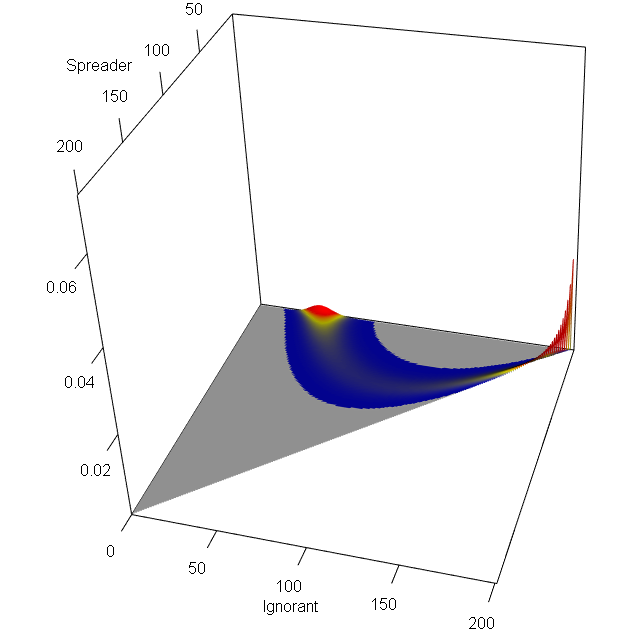}
		%\caption{1a}
	\end{subfigure}%
	\begin{subfigure}{.5\textwidth}
		\centering
		\includegraphics[width=.8\linewidth]{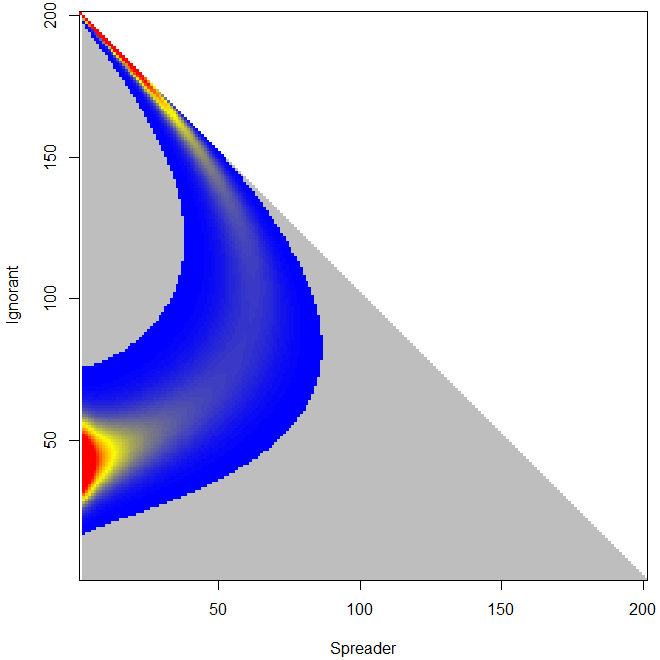}
		%\caption{1b}
	\end{subfigure}
	\caption{RE-distribution for the MT model with $(X_0=200,Y_0=1)$. Values bigger than \(10^{-3}\) are colored red, and values smaller than $10^{-6}$ are colored gray.}
	\label{Fig: RE}
\end{figure}

\section{DK model and  SIR model}
\label{S: SIR}

We end the discussion of this paper by analyzing the cases for the Daley--Kendall rumor model and the SIR epidemic model.
As we have pointed out, stochastic rumor processes and the SIR model remarkably differ in the way spreader/infected individuals are removed from the population.
However, many similarities exist: these processes are strictly evolutionary, conservative, and are absorbed in finite time. In this section, we discuss that the same strategy from Section \ref{SS: QSD Formula} can be applied to obtain a QSD for the DK model and the SIR model, but for the former, even if we adapt the absorbing set, there will exist a QSD in a very restricted setting, when the recovery rate is proportional to the population.
The proofs are analogous to the ones for the MT model and will be omitted.

\subsection{DK model}
\label{SS: QSD_DK}

Recall that, for the DK model, we also consider $X_t$, $Y_t$, and $Z_t$ representing the number of individuals in the ignorant, spreader, and stifler classes at time~$t$, respectively.
The population size is $N+1$, and the evolution is usually described by the continuous-time Markov chain $\{(X_t, Y_t)\}_{t \geq 0}$ with transition scheme~\eqref{F: DK Rates} and initial state $(N, 1)$. For this model, the total rate is given by
\begin{equation}
	\label{F: DK_total_rate}
	XY+Y(N+1-X-Y)+\frac{Y(Y-1)}{2} = \frac{Y(2N+1)}{2}-\frac{Y^2}{2}.
\end{equation}

Note that \eqref{F: DK_total_rate} has its maximum at $Y+N/2$, and since $Y \geq 0$, its minimum before absorption is reached at $Y=1$. Again we have
\begin{align*}
	S^{\prime} &= \{ (x, y) \in \bbZ^2: 0 \leq x \leq N, 0 \leq y \leq N + 1 - x \},\\[0.1cm]%
	\Delta &= \{(x, 0): x = 0, \dots, N - 1\},\\[0.1cm]%
	S &= S^{\prime} \setminus \Delta.
\end{align*}
Setting the absorption time
\[ \tam = \inf\{t > 0: Y_t = 0\}, \]
we can use the same arguments from Section \ref{SS: QSDs Rumor} to show that the QSD for the DK model is trivial. Note that the last state before absorption is the minimal class that can be reached from all previous states. 
So we can obtain a result analogous to Theorem \ref{T: QSD-MT}, and the QSD for the DK model is given by $\nu = \delta_{(0,1)}$.

Again, the only way to obtain a non-trivial QSD is to modify the conditioning by treating the set of states with $y=1$ as absorbing once the process leaves the initial state $(N, 1)$. Suppose that $(X_0, Y_0) = (N, 1)$, and let
\begin{equation*}
	% \label{F: initial_time_DK}
	\Xi = \inf\{t>0: Y_t \neq Y_0\}
\end{equation*}
Then we define a new absorption time by
\begin{equation}
	\label{F: abs_time_DK}
	\tam^* = \inf\{t > \Xi: Y_t = 1\},
\end{equation}
and the new set of transient states by
\begin{align*}
	\ctm &= S \setminus \{(x, 1): x = 0, \dots, N - 1\}\\[0.1cm]%
	&= \{(N, 1)\} \cup \{ (x, y) \in \bbZ^2: 0 \leq x \leq N, 2 \leq y \leq N + 1 - x \}.
\end{align*}

We then derive a theorem analogous to Theorem \ref{T: QSD-MT-2}, where we only get a non-trivial measure if we condition on not returning to the initial level. Otherwise, the QSD will be concentrated at a point.
The explicit formula for the DK model can be obtained similarly as before, with the difference that now we may have jumps of size two, so the number of paths is not related to the ballot problem. Since in this case paths to $(x,y)$ may have different sizes, we write $L_\gamma$ as the number of jumps of path $\gamma$. Definition \ref{D: Path} can be adapted for the DK model, where paths will have the possible transitions given by $\ctran = \{ (-1, 1), (0, -1),(0,-2) \}$.

Notice that again any path $\cam \in \ccam$ has $N-x$ transitions of type $(-1, 1)$, however, now the number of transitions of type $(0, -1)$ and $(0,-2)$ will depend on the path. Since we may have jumps of size two, notice that the number of jumps of this size is at least one, since the first time we lose spreaders must be a consequence of the encounter of two spreaders, considering we do not have any stifler yet, and the bigger number of jumps in this case is equal to \[\left\lfloor{\frac{z}{2}}\right\rfloor=\left\lfloor{\frac{N+1-x-y}{2}}\right\rfloor.\] For transitions of type $(0,-1)$, the size will be then \[\left\lceil{\frac{z}{2}}\right\rceil-\left\lceil{\frac{z-1}{2}}\right\rceil \ + \ 2j=\left\lceil{\frac{N+1-x-y}{2}}\right\rceil-\left\lceil{\frac{N-x-y}{2}}\right\rceil \ +\ 2j, \ j=0,\dots,\left\lfloor{\frac{N+1-x-y}{2}}\right\rfloor-1.\] 
Therefore, the sizes for paths to a fixed $(x,y)$ will be \[L_{\gamma} = N-x \ + \left\lfloor{\frac{N+1-x-y}{2}}\right\rfloor -j+ \left\lceil{\frac{N+1-x-y}{2}}\right\rceil-\left\lceil{\frac{N-x-y}{2}}\right\rceil+2j,\] 
for $j=0,\dots,\left\lfloor{\frac{N+1-x-y}{2}}\right\rfloor-1$. Then, since for $n \in \mathbb{Z}$ we have that $n=\lfloor\frac{n}{2}\rfloor+\lceil\frac{n}{2}\rceil$,
\begin{equation*}
	% \label{F: path_size_DK}
	L_{\gamma} = 2N-2x-y+1 -\left\lceil{\frac{N-x-y}{2}}\right\rceil+j.
\end{equation*}
Now, adapting Definition \ref{D: Rates-Path}, define 
\begin{equation*}
	% \label{F: rho_DK}
	\rho_j  = 
	\left\{
	\begin{array}{cl}
		x_j \, y_j &\text{if } v_{j+1} - v_j = (-1, 1),\\[0.2cm]
		y_j (N +1 - x_j-y_j) &\text{if } v_{j+1} - v_j = (0, -1).\\[0.2cm]
		y_j (y_j-1)/2 &\text{if } v_{j+1} - v_j = (0, -2).
	\end{array}	\right.   
\end{equation*}
Finally, we can state the following equivalent theorem for the DK model.

\begin{thm}
	\label{T: QSD_DK_3}
	Consider the DK model with absorption time defined by~\eqref{F: abs_time_DK}. 
	Then, the non-trivial QSD satisfying is given by
	\begin{equation*}
		% \label{F: formula_QSD_DK}
		\qsm(x,y) = 
		\left\{
		\begin{array}{cl}
			C_N^{-1} &\text{if } (x, y) = (N, 1),\\[0.1cm]
			C_N^{-1} \, \frac{\lambda_0}{\lambda_1-\lambda_0} &\text{if } (x, y) = (N-1, 2),\\[0.1cm]
			C_N^{-1} \displaystyle{\sum_{\cam \in \ccam}\frac{\lambda_0}{\lambda_L-\lambda_0} \prod_{j=1}^{L_\gamma-1} \dfrac{\rho_j}{\lambda_{j} - \lambda_0}} &\text{otherwise},
		\end{array}	\right. 
	\end{equation*}
	where $C_N$ is the normalizing constant defined in such a way that
	\[ \sum_{(x, y) \in \ctm} \qsm(x,y) = 1. \]
\end{thm}

As before, the QSD is given by a recurrent formula, as a function of the rates. Since now the paths have different sizes, the analysis is trickier than for the Maki--Thompson case. Note that 
\[\lambda_j-\lambda_0=\frac{Y(2N+1)}{2}-\frac{Y^2}{2}-N=\frac{(Y-1)(2N-Y)}{2}.\]
Let us analyze the value of $\qsm(x, y)$ with $y=N+1-x$, so that $L = y-1$ and $\rho_j = x_jy_j, \ j=0,\dots,L_\gamma$. Then,
\begin{align*}
	&\qsm(x,N+1-x) = \frac{\lambda_0}{\lambda_{L}-\lambda_0} \frac{\rho_1}{\lambda_1-\lambda_0}\frac{\rho_2}{\lambda_2-\lambda_0}\dots \frac{\rho_{L-2}}{\lambda_{L-2}-\lambda_0}\frac{\rho_{L-1}}{\lambda_{L-1}-\lambda_0}\\
	&= \frac{2N}{(y-1)(2N-y)} \cdot\frac{2\cdot2(N-1)}{(2N-2)}\cdot\frac{2\cdot3(N-2)}{2(2N-3)}\cdot\frac{2\cdot4(N-3)}{3(2N-4)}\dots\cdot\frac{2(y-1)(N-y+2)}{(y-2)(2N-y+1)}\\
	&=2^{N-x+1}\frac{N!}{(N-y+1)!} \frac{(2N-y-1)!}{(2N-2)!}= 2^{N-x+1}\frac{N!}{x!} \frac{(N+x-2)!}{(2N-2)!}.
\end{align*}
This can be rewritten as
\begin{equation}
	\label{F: qsm-approx}
	\qsm(x,N+1-x)=\frac{2^{2N}N!(N-2)!}{(2N-2)!} \binom{x+N-2}{x}
	\left(\frac{1}{2}\right)^x\left(\frac{1}{2}\right)^{N-1},
\end{equation}
where $\binom{x+N-2}{x} \left(\frac{1}{2}\right)^x\left(\frac{1}{2}\right)^{N-1}$ is the probability mass function at the point $x$ of a Negative Binomial random variable with parameters $r=N-1$ and $p=1/2$. Using Stirling's approximation for $N$ large, we see that
\begin{equation*}
	\frac{2^{2N}N!(N-2)!}{(2N-2)!} = \frac{2N-1}{N-1} \cdot \frac{2^{2N+1}N!N!}{(2N)!} \sim \frac{2N-1}{N-1} \cdot 2\sqrt{\pi N}.
\end{equation*}
Regarding the second factor in~\eqref{F: qsm-approx}, we observe that, for large enough $N$, it is approximately equal to the probability that a $\DNormal(N-1,2N-2)$ random variable lies in the interval $[x-1/2, x+1/2]$.  
Therefore,
\begin{align*}
	\qsm(x,N+1-x) &\sim \frac{2N-1}{N-1} \cdot 2\sqrt{\pi N}\cdot\frac{1}{\sqrt{4\pi (N-1)}}\exp\left\{\frac{-(x-N+1)^2}{4(N-1)}\right\}\\
	&\sim\frac{2N-1}{N-1}\sqrt{\frac{N}{N-1} }\exp\left\{\frac{-(x-N+1)^2}{4(N-1)}\right\}.   
\end{align*}
As in the case of the MT model, we see that the probability of staying on this path approaches zero really fast. Figure~\ref{Fig: DK_50} shows a plot for the QSD for the DK model with $N=50$. In this case, the distribution seems to be more concentrated close to the end of the process, and the curved shape is not observed. Note that a transition from $Y=2$ to $Y=3$ is not possible if there are no stifler individuals. This can be seen as the white diagonal in the 2D plot to the right of Figure~\ref{Fig: DK_50}.

\begin{figure}
	\begin{subfigure}{.5\textwidth}
		\centering
		\includegraphics[width=.8\linewidth]{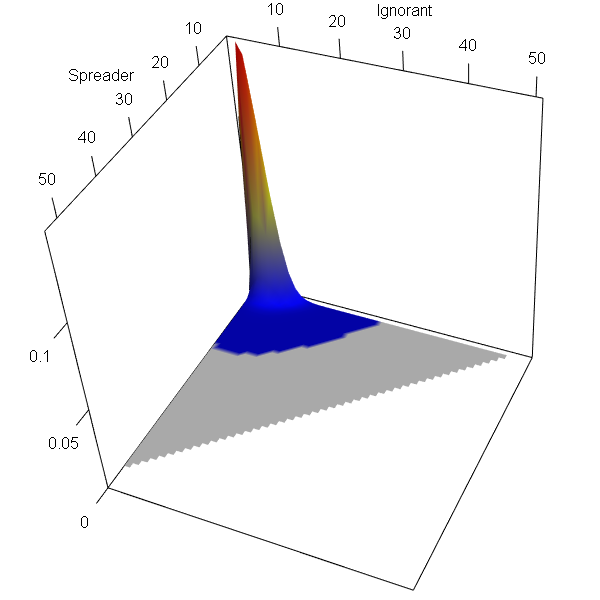}
		%\caption{1a}
	\end{subfigure}%
	\begin{subfigure}{.5\textwidth}
		\centering
		\includegraphics[width=.8\linewidth]{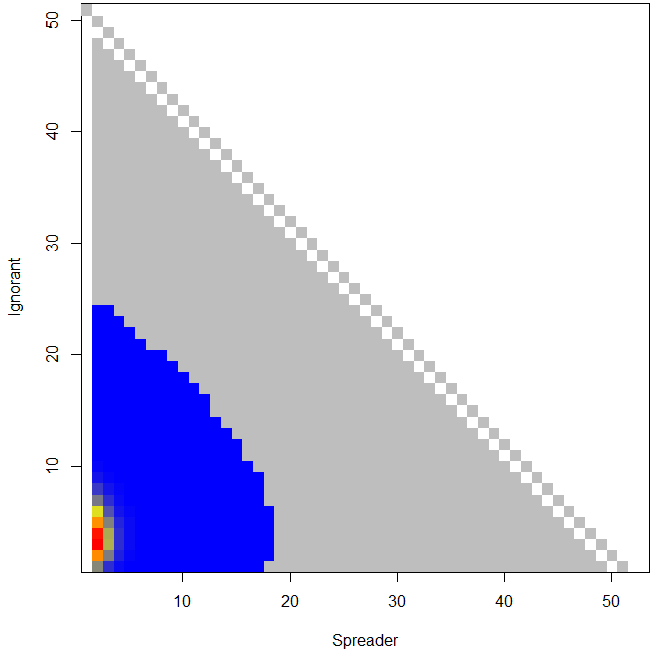}
		%\caption{1b}
	\end{subfigure}
	\caption{QSD $\qsm$ for the DK model with $(X_0=50, Y_0=1)$. Values smaller than $10^{-10}$ are colored gray.}
	\label{Fig: DK_50}
\end{figure}

\subsection{SIR model}
\label{SS: QSD_SIR}
Now let us consider the Markovian SIR epidemic model, also known as the \emph{general stochastic epidemic}, which was introduced by \textcite{bartlett49}.
We denote by $S_t$, $I_t$, and $R_t$ the number of susceptible, infected, and removed individuals at time $t \geq 0$.
Considering a population of size $N+1$, we have that $S_t + I_t + R_t = N+1$ for all $t \geq 0$.
The system evolves as a continuous-time Markov chain $\{(S_t, I_t)\}_{t \geq 0}$ with the following rates:
\begin{equation}
	\label{F: SIR_rates}
	\begin{array}{cc}
		\text{transition} \quad & \text{rate}\\[0.1cm] 
		(-1,1) \quad & \beta S I\\[0.1cm]
		(0,-1) \quad & \mu I.
	\end{array}     
\end{equation}
The first transition corresponds to an infection, and the second transition represents a removal; $\beta > 0$ and $\mu > 0$ are the infection and removal rates, respectively.
States $(s, i)$ with $i=0$ are absorbing.
For more details on this model and other stochastic processes for the spread of an infectious disease, we refer the reader to \textcite{allen03}, \textcite{andersson00}, and \textcite{daley99}.

We observe that the total rate at the state $(S, I)$ is $I(\beta S+\mu)$.
Furthermore, as proved by \textcite{artalejo10}, all probability mass for the QSD is concentrated at the state $(1, 0)$. They do not discuss it further, but the arguments to show this are the same ones used on rumor processes.
Next, we attempt to adapt this process to find a non-trivial QSD; however, the fact that individuals recover independently slightly alters the approach.

We first change the absorbing set in order to get a minimal class that accesses all other states, as done in \eqref{F: abs_time}. Then, suppose that $(S_0, I_0) = (N, 1)$, and let
\begin{equation*}
	% \label{F: initial_time_SIR}
	\hat{\Xi} = \inf\{t>0: I_t \neq I_0\}
\end{equation*}
be the first jump epoch of the SIR model.
Then we define a new absorption time by
\begin{equation}
	\label{F: abs_time_SIR}
	\tasir = \inf\{t > \Xi: I_t = 1\},
\end{equation}
and the new set of transient states by
\begin{equation*}
	\ctmsir = S \setminus \{(x, 1): x = 0, \dots, N - 1\}.
\end{equation*}

Note that we want $\lambda$ maximal to be associated with the state $(S_t =N,I_t=1)$ so all other states are accessible from the minimal class, as done in Section \ref{S: QS Rumor}. In this case, the rate of change associated with the minimal class should be $$I(\beta S+\mu)=\beta N+\mu.$$ 
However, note that the rates of change are minimal when $S=0$. Then, if we condition on not returning to $I=1$, the minimal rate will be equal to $\lambda=2\mu$, causing the QSD to be concentrated at the point $(S=0,I=2)$. This implies that to get a non-trivial QSD for the SIR model using the same strategy as before, we need that
\begin{equation*}
	2 \mu > \beta N + \mu \Longrightarrow \mu > \beta N.
\end{equation*}
This means that the rate of recovery $\mu$ needs to be much higher than the infection rate, and to be proportional to the size of the population for the QSD to be non-trivial. We summarize this in the following theorem.

\begin{thm}\label{T: QSD_SIR}
	Consider the hitting time defined in \eqref{F: abs_time_SIR} for the SIR epidemic model with rates \eqref{F: SIR_rates}. Then there is a non-trivial measure $\qsir$ such that
	\begin{equation*}
		P_{\qsir}(S_t=s,I_t=i \dado \tasir >t)=\qsir(s,i),
	\end{equation*}
	if, and only if, $\mu > \beta N$. In this case, $\qsir(s,i)>0$ for all transient states. This QSD is the limiting measure
	\begin{equation*}
		\lim_{t \to \infty}P_{(N,1)}(S_t=x,I_t=y)=\qsir(s,i)  
	\end{equation*}
	when $\mu=\delta_{(N,1)}$ is the initial distribution. For any other initial distribution $\rho$, if $\rho(N,1)=0$,
	\begin{equation*}
		\lim_{t \to \infty}P_{\rho}(X_t=x,I_t=y)=\delta_{(0,2)}.  
	\end{equation*}
\end{thm}

The proof is analogous to the proof of Theorem \ref{T: QSD_MT_3}. Individuals recovering independently in this model is the crucial difference that may make it inadequate for the use of QSDs, since the scenario that allows us to get a QSD is very restricted. The RE-distribution is possibly a better approximation for a long-term behavior, as discussed in \textcite{artalejo10}. 
The explicit QSD for this model can be found similarly to the MT model in Section \ref{SS: QSD Formula}.

\printbibliography

\end{document}